\documentclass[10pt,reqno]{amsart}
\usepackage{amsmath,amsopn,amssymb,amsthm}
\usepackage[cp1251]{inputenc}
\usepackage[english]{babel}
\usepackage[pagebackref,breaklinks=true,colorlinks=true,linkcolor=blue,citecolor=blue,urlcolor=blue]{hyperref}
\usepackage{graphicx}%
\usepackage{color}

\newtheorem{theorem}{Theorem}[section]

\newtheorem{lemma}[theorem]{Lemma}

\renewenvironment{proof}[1][Proof]{\noindent\textbf{#1.} }{\ \rule{0.5em}{0.5em}}

\begin{document}

\title{Randers and $(\alpha,\beta)$ equigeodesics for some compact homogeneous manifolds}
\author{Ju Tan}
\address[Ju Tan]{School of Mathematics and Physics,
Anhui University of Technology, Maanshan, 243032, P.R. China}\email{tanju2007@163.com}
\author{Ming Xu$^*$}
\address[Ming Xu] {School of Mathematical Sciences,
Capital Normal University,
Beijing 100048,
P.R. China}
\email{mgmgmgxu@163.com}

\thanks{$^*$Ming Xu is the corresponding author.}

\begin{abstract}
A smooth curve on $G/H$ is called a Riemannian equigeodesic if it is a homogeneous geodesic for all $G$-invariant Riemannian metrics on $G/H$. With the $G$-invariant Riemannian metric replaced by other classes of $G$-invariant metrics, we can similarly define Finsler equigeodesic, Randers equigeodesic, $(\alpha,\beta)$ equigeodesic, etc.
In this paper, we study Randers and $(\alpha,\beta)$ equigeodesics. For a compact homogeneous manifold, we prove Randers and $(\alpha,\beta)$ equigeodesics are equivalent, and find a criterion for them. Using this criterion
we can classify the equigeodesics on many compact homogeneous manifolds which permit non-Riemannian homogeneous Randers metrics, including four classes of homogeneous spheres.
\vbox{}
\vbox{} 
\end{abstract}

\maketitle

\section{Introduction}
\subsection{Homogeneous geodesic and equigeodesic in Riemannian and Finsler geometry}
On a Riemannian manifold, a geodesic is called {\it homogeneous} if it is the orbit of a one-parameter subgroup of isometries. This notion reflects a key philosophy in Erlanger Programme, i.e., applying Lie theory to the study of geometry.
Many important topics in homogeneous Riemannian geometry are closely related to the study of homogeneous geodesic. For example, O. Kowalski and L. Vanhecke defined a Riemannian manifold to be {\it geodesic orbit} (or {\it g.o.} in short), if each geodesic is homogeneous \cite{KV1991}. Both g.o. spaces and their subclasses, normal homogeneous \cite{Be1961,WZ1985} and $\delta$-homogeneous spaces \cite{BN2008}, weakly symmetric spaces \cite{Zi1996}, etc, have been extensively studied in Riemannian geometry \cite{AA2007,AN2009,DKN2004,Go1996,GN2018,Ni2013,WZ1985}. Recently, they have been generalized and studied in Finsler geometry \cite{De2012,La2007,Xu2018,Xu2021,XD2017,YD2014,ZX2022}.

Equigeodesic was introduced by N. Cohen, L. Grama and C.J.C. Negreiros in 2010 \cite{CGN2010}.
A smooth curve on a homogeneous space $G/H$ is called an {\it equigeodesic} if it is
a homogeneous geodesic for any $G$-invariant metric. When the type of the metric is specified,
Riemannian, Finsler, Randers, $(\alpha,\beta)$, etc,
the corresponding equigeodesic is called a {\it Riemannian equigeodesic}, {\it Finsler equigeodesic}, {\it Randers equigeodesic}, {\it $(\alpha,\beta)$ equigeoesic}, etc, respectively. Until now, only Riemannian equigeodesic has been studied for some special homogeneous spaces, flag manifolds, spheres, etc \cite{GN2011,St2020,WZ2013}. Finsler and other types of equigeodesic are also worthy to be studied because they represent different algebraic properties.
\subsection{Reduction to equigeodesic vector}
A vector $X\in\mathfrak{g}$ is called a {\it Riemannian}
equigeodesic vector for $G/H$ if it generates a Riemannian equigeodesic $c(t)=\exp tX\cdot o$
passing $o=eH$.

By homogeneity, the study for Riemannian equigeodesic
can be reduced to that for Riemannian equigeodesic vector.
Further reduction need more assumptions for $G/H$.

In this paper, we only consider a homogeneous manifold $G/H$ with a compact $(G,\langle\cdot,\cdot\rangle_{\mathrm{bi}})$, in which
$\langle,\rangle_{\mathrm{bi}}$ is an  $\mathrm{Ad}(G)$-invariant
inner product on
$\mathfrak{g}$. We denote $\mathfrak{g}=\mathfrak{h}+\mathfrak{m}$ the corresponding orthogonal reductive decomposition, and use subscripts $\mathfrak{h}$ and $\mathfrak{m}$ for projections to each subspace.
Then
any $G$-invariant Riemannian
on $G/H$ is one-to-one determined by an $\mathrm{Ad}(H)$-invariant inner product
$$\alpha(u,v)=\langle u,\Lambda(v)
\rangle_{\mathrm{bi}},\quad\forall u,v\in\mathfrak{m},$$
in which $\Lambda:\mathfrak{m}\rightarrow\mathfrak{m}$ is called the {\it metric operator}.

By the criterion for geodesic vector \cite{KV1991},   $X\in\mathfrak{g}$ is
a Riemannian equigeodesic for $G/H$ if and only if $X\notin\mathfrak{h}$ and
\begin{equation}\label{014}
[\Lambda(X_\mathfrak{m}),X]_\mathfrak{m}=0,\quad\forall \Lambda,
\end{equation}
where the metric operator $\Lambda$ exhausts
all $\mathrm{Ad}(H)$-invariant $\langle\cdot,\cdot\rangle_{\mathrm{bi}}$-positive definite
linear endomorphisms on $\mathfrak{m}$. In particular, we can take $\Lambda=\mathrm{Id}$, then (\ref{014}) implies, when $X$ is a Riemannian equigeodesic vector, $X_\mathfrak{h}$ and $X_\mathfrak{m}$ commute, and then $X_\mathfrak{m}$ is
also a Riemannian equigeodesic vector which generates the same curve $c(t)=\exp tX_\mathfrak{m}\cdot o=\exp tX\cdot o$ as $X$.

To summarize, the study of (Riemannian) equigeodesics for $G/H$ can be reduced to that of
 (Riemannian) equigeodesic vectors which are contained in $\mathfrak{m}$.
{\it Finsler}, {\it Randers},  $(\alpha,\beta)$ (and other types of) {\it equigeodesic vectors} can be defined similarly, i.e., $X\in\mathfrak{g}$ which generates an equigeodesic $c(t)=\exp tX\cdot o$ of the specified type. Obviously they are also Riemannian equigeodesic vectors. So above reduction is still valid.

\subsection{Main results}
In this paper, we mainly discuss Randers and $(\alpha,\beta)$ equigeodesics. Firstly, we prove they are equivalent for a compact homogeneous manifold.

\begin{theorem}\label{main-thm-1}
Let $G/H$ be a homogeneous manifold with a compact Lie group $(G,\langle\cdot,\cdot\rangle_{\mathrm{bi}})$ and the corresponding orthogonal reductive decomposition
$\mathfrak{g}=\mathfrak{h}+\mathfrak{m}$. Suppose that $\mathfrak{m}$ can be further $\mathrm{Ad}(H)$-invariantly decomposed as $\mathfrak{m}=\mathfrak{m}_0+\mathfrak{m}'$, in which $\mathfrak{m}_0\neq0$ is the fixed point set for the $\mathrm{Ad}(H)$-action on $\mathfrak{m}$. Then for any $X\in\mathfrak{m}\backslash\{0\}$, the following statements are equivalent:
\begin{enumerate}
\item $X$ is an
$(\alpha,\beta)$ equigeodesic vector;
\item $X$ is a Randers equigeodesic vector;
\item $X$ is a Riemannian equigeodesic vector and $[X,\mathfrak{m}_0]\subset\mathfrak{h}$.
\end{enumerate}
\end{theorem}

As applications of Theorem \ref{main-thm-1}, we can classify Randers and $(\alpha,\beta)$ equigeodesic vectors on some compact homogeneous manifolds.
 To avoid iteration, $(\alpha,\beta)$ equigeodesic or
$(\alpha,\beta)$ equigeodesic vector will not be mentioned later.
We only concern those homogeneous manifolds which admits non-Riemannian homogeneous Randers metrics, otherwise there is no new story.

\begin{theorem}\label{main-thm-2}
Let $H\subsetneqq K\subsetneqq (G,\langle\cdot,\cdot\rangle_{\mathrm{bi}})$ be a triple of distinct compact connected Lie groups, such that $G/K$ is  strongly isotropy irreducible, and $H$ is normal in $K$.
Denote $\mathfrak{m}_0$ and $\mathfrak{m}'$ the $\langle\cdot,\cdot\rangle_{\mathrm{bi}}$-orthogonal complements of $\mathfrak{h}$ in $\mathfrak{k}$ and $\mathfrak{k}$ in $\mathfrak{g}$ respectively. Suppose that neither $\mathfrak{h}$ nor $\mathfrak{m}_0$ is an ideal of $\mathfrak{g}$.
Then the set of all Randers equigeodesic vectors for $G/H$ which are contained in $\mathfrak{m}=\mathfrak{m}_0+\mathfrak{m}'$ is  $\mathfrak{c}(\mathfrak{m}_0)\backslash\{0\}$.
\end{theorem}

%a homogeneous manifold with compact $(G,\langle\cdot,\cdot\rangle_{\mathrm{bi}})$ and connected $H$. Denote $\mathfrak{g}=\mathfrak{h}+\mathfrak{m}$ the corresponding orthogonal reductive decomposition and $\mathfrak{m}_0=\mathfrak{c}_\mathfrak{m}(\mathfrak{h})$
%the centralizer of $\mathfrak{h}$ in $\mathfrak{m}$. Suppose neither $\mathfrak{h}$ or $\mathfrak{m}_0$ is an ideal of $\mathfrak{g}$.
%Then the set of Randers  equigeodesic vectors for $G/H$ is $\mathfrak{c}(\mathfrak{m}_0)$.
%\end{theorem}

By the classification of compact irreducible symmetric spaces
\cite{He1962} and compact strongly isotropy irreducible
spaces \cite{Ma1961-1,Ma1961-2,Ma1966,Wo1968}, almost all $G/H$'s in Theorem \ref{main-thm-2} have no Randers equigeodesics or Randers equigeodesic vectors, except
those associated with compact irreducible Hermitian symmetric $G/K=G/(U(1)H)$, i.e.,
\begin{eqnarray}\label{011}
G/H&=&SU(n_1+n_2)/SU(n_1)SU(n_2),\
Sp(n)/SU(n),\
SO(n+2)/SO(n),\nonumber\\
& &SO(2n)/SU(n),\
E_6/SO(10), \
E_7/E_6,
\end{eqnarray}
have a line $\mathfrak{m}_0\backslash\{0\}$
of Randers equigeodesic vectors.
 As Theorem \ref{main-thm-2} has not assumed any effectiveness for the $G$-action on $G/K$, it can also tell us that the set of Randers equigeodesic vectors for $U(1)G/U(1)H$ with $G/H$ in (\ref{011}), which are contained in $\mathfrak{m}$, is the line
$\mathfrak{m}_0\backslash\{0\}$.

Finally, we give a classification of Randers equigeodesics for homogeneous spheres which permit non-Riemannian homogeneous Randers metrics.

\begin{theorem}\label{main-thm-3}
Let $G/H$ be a homogeneous sphere on which the compact connected Lie group $(G,\langle,\rangle_{\mathrm{bi}})$ acts effectively. Denote $\mathfrak{g}=\mathfrak{h}+\mathfrak{m}$ the corresponding orthogonal reductive decomposition and
assume $G/H$ admits non-Riemannian $G$-invariant Randers metrics, i.e.,  the fixed point set $\mathfrak{m}_0$ for the $\mathrm{Ad}(H)$-action on $\mathfrak{m}$ is not zero.
Then we have the following:
\begin{enumerate}
\item
the set of Randers equigeodesic vectors for $G/H$ which is contained in $\mathfrak{m}$ is $\mathfrak{m}_0\backslash\{0\}$ when $G/H=SU(n+1)/SU(n)$, $U(n+1)/U(n)$ or $Sp(n+1)U(1)/Sp(n)U(1)$;
\item
$G/H=Sp(n+1)/Sp(n)$ has no Randers equigeodesics or Randers equigeodesic vectors.
\end{enumerate}
\end{theorem}

These results implies that studying equigeodesic in a  Finsler context might be even easier sometimes, where cleaner classification lists would be more achievable.

This paper is organized as following. In Section 2, we summarize some necessary knowledge on homogeneous
Randers and $(\alpha,\beta)$ spaces and prove Theorem \ref{main-thm-1}. In Section 3, we prove Theorem \ref{main-thm-2} and Theorem \ref{main-thm-3} as applications of Theorem \ref{main-thm-1}.

\section{Randers and $(\alpha,\beta)$ equigeodesic}
\subsection{Homogeneous Randers and $(\alpha,\beta)$ spaces}

Let $G/H$ be a homogeneous manifold
with a reductive decomposition $\mathfrak{g}=\mathfrak{h}+\mathfrak{m}$.
We decompose $\mathfrak{m}$  $\mathrm{Ad}(H)$-equivariantly as
$\mathfrak{m}=\mathfrak{m}_0+\mathfrak{m}'$, in which $\mathfrak{m}_0$
is the fixed point set for the $\mathrm{Ad}(H)$-action on $\mathfrak{m}$.

Any $G$-invariant $(\alpha,\beta)$ metric $F$ on $G/H$ is one-to-one determined by the $\mathrm{Ad}(H)$-invariant
Minkowski norm in the tangent space at $o=eH$, which can be presented as
\begin{equation}\label{000}
F(y)=\alpha(y,y)^{1/2}
\varphi(\tfrac{\alpha(y,u)}{\alpha(y,y)^{1/2}}).
\end{equation}
Here $\alpha$ is an $\mathrm{Ad}(H)$-invariant inner product on $\mathfrak{m}$, $u$ is a nonzero vector in $\mathfrak{m}_0$ and
$\varphi=\varphi(s)$ is a positive smooth function for $|s|\leq
b=\alpha(u,u)^{1/2}$ with the inequality
$$\varphi(s)-s\varphi'(s)+(b^2-s^2)\varphi''(s)>0$$
satisfied everywhere. In particular, this homogeneous metric $F$ is Randers when $\varphi(s)=1+s$, which is of the form
\begin{equation}\label{007}
F(y)=\alpha(y,y)^{1/2}+\alpha(y,u).
\end{equation}

A vector $X\in\mathfrak{g}$ is called a {\it geodesic vector} for $(G/H,F)$ if $c(t)=\exp tX\cdot o$ is a homogeneous geodesic. When $F$ is
$(\alpha,\beta)$ or Randers, we have the following criterion for geodesic vectors (see Proposition 3.4 in \cite{Ya2017}).

\begin{lemma}\label{lemma-1}
Keep all assumptions and notations in this subsection. Then $X\in\mathfrak{g}$ is a geodesic vector of the homogeneous $(\alpha,\beta)$
metric $F$ in (\ref{000}) if and only if $X\notin\mathfrak{h}$ and
\begin{equation}\label{001}
\alpha([X,Z]_\mathfrak{m},
(\varphi(s)-s\varphi'(s))X_\mathfrak{m}
+\varphi'(s)
\alpha(X_\mathfrak{m},X_\mathfrak{m})^{1/2}u)=0,\quad
\forall Z\in\mathfrak{m},
\end{equation}
where $s=\tfrac{\alpha(X_\mathfrak{m},u)}{
\alpha(X_\mathfrak{m},X_\mathfrak{m})^{1/2}}$.
In particular, when $\varphi(s)=1+s$, (\ref{001}) can be simplified as
\begin{equation}\label{002}
\alpha([X,Z]_\mathfrak{m},X_\mathfrak{m}+
\sqrt{\alpha(X_\mathfrak{m},X_\mathfrak{m})}u)=0.
\end{equation}
\end{lemma}

See \cite{BCS2000,De2012,YD2014} for
more details on general and homogeneous Finsler geometry, as well as the related homogeneous geodesic and g.o. property.
\subsection{Proof of Theorem \ref{main-thm-1}}

Firstly, the argument from (1) to (2) is obvious because each Randers metric is an $(\alpha,\beta)$ metric.

Nextly, we prove the statement from (2) to (3). Let $\alpha(\cdot,\cdot)=\langle\cdot,\Lambda(\cdot)
\rangle_{\mathrm{bi}}$ be any $\mathrm{Ad}(H)$-invariant inner product on $\mathfrak{m}$, and $u$ any vector in $\mathfrak{m}_0$. Then the Rander equigeodesic vector $X\in\mathfrak{m}\backslash\{0\}$ is a geodesic vector for the homogeneous Randers metric
$F(y)=\alpha(y,y)^{1/2}+\alpha(y,u)$. Applying (\ref{002})
in Lemma \ref{lemma-1}, we get
\begin{equation}\label{003}
\alpha([X,Z]_\mathfrak{m},X)+\alpha(X,X)^{1/2}
\alpha([X,Z]_\mathfrak{m},u)=0.
\end{equation}

Obviously the Randers equigeodesic vector $X$
is a Riemannian equigeodesic vector, i.e.,
$\alpha([X,\mathfrak{m}]_\mathfrak{m},X)=0$.
So we have
\begin{eqnarray*}
0=\alpha([X,\mathfrak{m}]_\mathfrak{m},u)=
\langle[X,\mathfrak{m}],\Lambda(u)\rangle_{\mathrm{bi}}
=\langle\mathfrak{m},[X,\Lambda(u)]\rangle_{\mathrm{bi}},
\end{eqnarray*}
i.e., $[X,\Lambda(u)]\subset\mathfrak{h}$.
By Schur lemma, the metric operator $\Lambda$ for $\alpha$ must be of the form $\Lambda=\Lambda_0\oplus\Lambda'$ in which $\Lambda_0$ is any arbitrary $\langle\cdot,\cdot\rangle_{\mathrm{bi}}$-positive definite linear endomorphism on $\mathfrak{m}_0$ and $\Lambda'$ is some endomorphism on $\mathfrak{m}'$. As we have assumed
$u\neq0$, $\Lambda(u)=\Lambda_0(u)$ exhausts an open neighborhood of $u$ in $\mathfrak{m}_0$. So we get
$[X,\mathfrak{m}_0]\subset\mathfrak{h}$.

This ends the proof from (2) to (3).

Finally, we prove the statement from (3) to (1). Let $F$ be any $G$-invariant $(\alpha,\beta)$ metric on $G/H$, determined by
the Minkowski norm
$F(y)=\alpha(y,y)^{1/2}\varphi(\tfrac{\alpha(y,u)}{
\alpha(y,y)^{1/2}})$ on $\mathfrak{m}$. Since the Riemannian equigeodesic vector $X$ is a geodesic vector for the $G$-invariant Riemannian metric $\alpha$, we have
\begin{equation}\label{004}
\alpha([X,Z]_\mathfrak{m},X)=0,\quad\forall Z\in\mathfrak{m}.
\end{equation}
As we have observed, the metric operator $\Lambda$ of $\alpha$ preserves $\mathfrak{m}_0$. So we have
\begin{eqnarray}\label{005}
\alpha([X,\mathfrak{m}]_\mathfrak{m},u)=
\langle[X,\mathfrak{m}],
\Lambda(u)\rangle_{\mathrm{bi}}=
\langle \mathfrak{m},[X,\Lambda(u)]\rangle_{\mathrm{bi}}
\subset\langle\mathfrak{m},[X,\mathfrak{m}_0]\rangle_{\mathrm{bi}}
\subset\langle\mathfrak{m},\mathfrak{h}\rangle_{\mathrm{bi}}=0.
\end{eqnarray}
Summarizing (\ref{004}) and (\ref{005}), we get (\ref{001}) in Lemma \ref{lemma-1}, i.e., $X$ is a geodesic vector for $F$.
Since the $G$-invariant $(\alpha,\beta)$-metric $F$ is arbitrarily chosen,
$X$ is an $(\alpha,\beta)$ equigeodesic vector.

This ends the proof from (3) to (1).
\section{Applications of Theorem \ref{main-thm-1}}
\subsection{Proof of Theorem \ref{main-thm-2}}
We keep all relevant assumptions and notations for Theorem
\ref{main-thm-2}.
Then $\mathfrak{g}=\mathfrak{k}+\mathfrak{m}'$ and
$\mathfrak{g}=\mathfrak{h}+\mathfrak{m}=\mathfrak{h}
+(\mathfrak{m}_0+\mathfrak{m}')$ are
$\langle\cdot,\cdot\rangle_{\mathrm{bi}}$-orthogonal
reductive decompositions for $G/K$ and $G/H$ respectively.
Together with the ideal decomposition $\mathfrak{k}=\mathfrak{h}\oplus\mathfrak{m}_0$, they provide the following bracket relations:
\begin{equation}\label{006}
[\mathfrak{h},\mathfrak{m}' ]\subset\mathfrak{m}',\quad
[\mathfrak{m}_0,\mathfrak{m}']\subset\mathfrak{m}',\quad
[\mathfrak{h},\mathfrak{m}_0]=0,\quad [\mathfrak{h},\mathfrak{h}]\subset\mathfrak{h},\quad
[\mathfrak{m}_0,\mathfrak{m}_0]\subset\mathfrak{m}_0.
\end{equation}
It implies that the decomposition $\mathfrak{m}=\mathfrak{m}_0+\mathfrak{m}'$ is $\mathrm{Ad}(H)$-invariant.
Because $\mathfrak{m}_0$ is not
an ideal of $\mathfrak{g}$, $\mathfrak{m}_0\neq0$, i.e.,
$G/H$ admits $G$-invariant non-Riemannian Randers metrics.
The next lemma implies that $\mathfrak{m}_0$ is the fixed point set for the $\mathrm{Ad}(H)$-action on $\mathfrak{m}$.
%Notice that
%The Cartan decomposition for the compact Hermitian symmetric space $G/K$ is .

\begin{lemma} \label{lemma-2}
Any vector $v\in\mathfrak{m}'$ satisfying
$[v,\mathfrak{h}]=0$ or $[v,\mathfrak{m}_0]=0$ must be 0.
\end{lemma}
\begin{proof}
Firstly, we prove $v\in\mathfrak{m}'$ satisfying $[v,\mathfrak{m}_0]=0$ must vanish.
Assume conversely $v\neq0$, then (\ref{006}) implies that the nonzero subspace $\mathfrak{c}_{\mathfrak{m}'}(\mathfrak{m}_0)
\subset\mathfrak{m}'$
is $\mathrm{ad}(\mathfrak{k})$-invariant. By the strongly isotropy irreducibility of $G/K$, we must have $\mathfrak{c}_{\mathfrak{m}'}(\mathfrak{m}_0)=
\mathfrak{m}'$, i.e., $[\mathfrak{m}',\mathfrak{m}_0]=0$. Together with
$[\mathfrak{k},\mathfrak{m}_0]\subset\mathfrak{m}_0$, it implies that $\mathfrak{m}_0$ is an ideal of  $\mathfrak{g}$. This contradicts the assumption in Theorem \ref{main-thm-2}.

Nextly, we prove $v\in\mathfrak{m}'$ satisfying $[v,\mathfrak{h}]=0$ must vanish. Assume conversely that it does not, then similar argument shows that $\mathfrak{c}_{\mathfrak{m}'}(\mathfrak{h})=
\mathfrak{m}'$.
Then $\mathfrak{h}$ is an ideal of $\mathfrak{g}$, but this contradicts the assumption in Theorem \ref{main-thm-2}.
\end{proof}

\begin{proof}[Proof of Theorem \ref{main-thm-2}]
Firstly, we prove that any $X\in\mathfrak{c}(\mathfrak{m}_0)\backslash\{0\}$ is
a Randers  equigeodesic vector. Lemma \ref{lemma-2} indicates that $\mathfrak{m}'$ does not contain any trivial sub-representation for the $\mathrm{Ad}(H)$-action. By Schur Lemma, the metric operator $\Lambda$ for a $G$-invariant Riemannian metric on $G/H$ must be of the form $\Lambda=\Lambda_0\oplus\Lambda'$, in which
each $\Lambda_0$ and $\Lambda'$ are linear endomorphisms on $\mathfrak{m}_0$ and $\mathfrak{m}'$ respectively. So for every metric operator $\Lambda$, we have
$$[\Lambda(X),X]_\mathfrak{m}\subset
[\mathfrak{m}_0,\mathfrak{c}(\mathfrak{m}_0)
]_\mathfrak{m}=0,$$
which proves $X$ is a Riemannian equigeodesic vector. Meanwhile, it is obvious that
$[X,\mathfrak{m}_0]=0
\subset\mathfrak{h}$. By Theorem \ref{main-thm-1}, $X\in\mathfrak{c}(\mathfrak{m}_0)\backslash\{0\}$ is a Randers  equigeodesic vector.

Nextly, we prove there are no other Randers  equigeodesic vectors. Suppose $X=X_0+X'\in\mathfrak{m}\backslash\{0\}$ with  $X_0\in\mathfrak{m}_0$ and $X'\in\mathfrak{m}'$
is a Randers  equigeodesic vector. By Theorem \ref{main-thm-1},
\begin{equation}\label{008}
[X,v]=[X_0,v]+[X',v]\in\mathfrak{h}, \quad\forall
v\in\mathfrak{m}_0.
\end{equation}
However, (\ref{006}) indicates that $[X_0,v]+[X',v]\in\mathfrak{m}_0+\mathfrak{m}'$ for $v\in\mathfrak{m}_0$.
So we must have
$[X_0,\mathfrak{m}_0]=[X',\mathfrak{m}_0]=0$, i.e.,
$X_0\in\mathfrak{c}(\mathfrak{m}_0)$, and by Lemma \ref{lemma-2}, $X'=0$.

This ends the proof of Theorem \ref{main-thm-2}.
\end{proof}
\label{section-3-1}

\subsection{Proof of Theorem \ref{main-thm-3}}

The classification for homogeneous spheres \cite{Bo1940,MS1943} can be summarized as Table \ref{table-1}, where we list the decomposition of each $\mathfrak{m}$, where $\mathfrak{m}_0\neq0$ is a trivial $H$-representation, and $\mathfrak{m}_i$'s with $i>0$ are distinct nontrivial irreducible $H$-representations.

%We summarize the results in \cite{} to Table \ref{table-1}, which classifies, in the local sense or in the Lie algebraic level, all
%homogeneous spheres $G/H$ on which the compact Lie group acts almost effectively. The decompositions for their isotropy representations are also shown in this table, where
%$\mathfrak{m}_0\neq0$ is the fixed point set of the $\mathrm{Ad}$-action on $\mathfrak{m}$, corresponding to the trivial representation, and $\mathfrak{m}_i$'s with $i>0$ in each $\mathfrak{m}$ are distinct nontrivial irreducible representations of $H$.

\begin{table}
  \centering
  \begin{tabular}{|c|c|c|c|c|}
     \hline
 No. &Spheres &$G$ &$H$ &Isotropy rep.\\
 \hline
1 &$\mathrm{S}^n$ &${SO}(n+1)$ &${SO}(n)$ &irred\\
2 &$\mathrm{S}^{2n+1}$ &${SU}(n+1)$ &${SU}(n)$ &$\mathfrak m=\mathfrak m_0\oplus\mathfrak m_1$\\
3 &$\mathrm{S}^{2n+1}$ &$U(n+1)$ &$U(n)$ &$\mathfrak m=\mathfrak m_0\oplus\mathfrak m_1$\\
4 &$\mathrm{S}^{4n+3}$ &${Sp}(n+1)$ &${Sp}(n)$ &$\mathfrak m=\mathfrak m_0\oplus\mathfrak m_1$\\
5 &$\mathrm{S}^{4n+3}$ &${Sp}(n+1){Sp}(1)$ &${Sp}(n){Sp}(1)$ &$\mathfrak m=\mathfrak m_1\oplus\mathfrak m_2$\\
6 &$\mathrm{S}^{4n+3}$ &${Sp}(n+1)U(1)$ &${Sp}(n)U(1)$ &$\mathfrak m=\mathfrak m_0\oplus\mathfrak m_1\oplus\mathfrak m_2$\\
7 &$\mathrm{S}^{15}$ &$\mathrm{Spin}(9)$ &$\mathrm{Spin}(7)$ &$\mathfrak m=\mathfrak m_1\oplus\mathfrak m_2$\\
8 &$\mathrm{S}^7$ &$\mathrm{Spin}(7)$ &$G_2$ &irred\\
9 &$\mathrm{S}^7$ &$G_2$ &${SU}(3)$ &irred\\
 \hline
   \end{tabular}
  \caption{Classification for homogeneous spheres}\label{table-1}
\end{table}

%\begin{longtable}{|c|c|c|c|c|}
%\label{table-1}
%\caption{Classification for homogeneous spheres}\\
%\hline
% No. &Spheres &$G$ &$H$ &Isotropy rep.\\
% \hline
%1 &$\mathrm{S}^n$ &${SO}(n+1)$ &${SO}(n)$ &irred\\
%2 &$\mathrm{S}^{2n+1}$ &${SU}(n+1)$ &${SU}(n)$ &$\mathfrak m=\mathfrak m_0\oplus\mathfrak m_1$\\
%3 &$\mathrm{S}^{2n+1}$ &$U(n+1)$ &$U(n)$ &$\mathfrak m=\mathfrak m_0\oplus\mathfrak m_1$\\
%4 &$\mathrm{S}^{4n+3}$ &${Sp}(n+1)$ &${Sp}(n)$ &$\mathfrak m=\mathfrak m_0\oplus\mathfrak m_1$\\
%5 &$\mathrm{S}^{4n+3}$ &${Sp}(n+1){Sp}(1)$ &${Sp}(n){Sp}(1)$ &$\mathfrak m=\mathfrak m_1\oplus\mathfrak m_2$\\
%6 &$\mathrm{S}^{4n+3}$ &${Sp}(n+1)U(1)$ &${Sp}(n)U(1)$ &$\mathfrak m=\mathfrak m_0\oplus\mathfrak m_1\oplus\mathfrak m_2$\\
%7 &$\mathrm{S}^{15}$ &$\mathrm{Spin}(9)$ &$\mathrm{Spin}(7)$ &$\mathfrak m=\mathfrak m_1\oplus\mathfrak m_2$\\
%8 &$\mathrm{S}^7$ &$\mathrm{Spin}(7)$ &$G_2$ &irred\\
%9 &$\mathrm{S}^7$ &$G_2$ &${SU}(3)$ &irred\\
% \hline
%\end{longtable}

We only concern those with $\mathfrak{m}_0$, which admits
non-Riemannian $G$-invariant Randers metrics. So $G/H$ is
one of the following:
\begin{eqnarray*}
& &S^{2n+1}=U(n+1)/U(n),\ S^{2n+1}=SU(n+1)/SU(n),\\
& &S^{4n+3}=Sp(n+1)/Sp(n),\   S^{4n+3}=Sp(n+1)U(1)/Sp(n)U(1).
\end{eqnarray*}

{\bf Case 1}: $G/H=SU(n+1)/SU(n)$ or $U(n+1)/U(n)$.

In this case, $\dim\mathfrak{m}_0=1$ and
$G/H$ is associated with the Hermitian symmetric space $G/K=SU(n+1)/S(U(n)U(1))$
or $U(n+1)/(U(n)U(1))$, Theorem \ref{main-thm-2} tells us
that the set of Randers equigeodesic vectors which are contained in $\mathfrak{m}$ is the line $\mathfrak{m}_0\backslash\{0\}$.

{\bf Case 2}: $G/H=Sp(n+1)/Sp(n)$.

In this case $G/H$ is associated with the quaternion symmetric $G/K=Sp(n+1)/Sp(n)Sp(1)$, and $\mathfrak{m}_0$ is the Lie subalgebra $sp(1)$. By Theorem \ref{main-thm-2}, the set of Randers equigeodesic vectors for $G/H$ which is contained in $\mathfrak{m}$ is  $\mathfrak{c}(\mathfrak{m}_0)\backslash\{0\}=\emptyset$.

{\bf Case 3}: $G/H=Sp(n+1)U(1)/Sp(n)U(1)$.

In this case, the decomposition $\mathfrak{m}=\mathfrak{m}_0+\mathfrak{m}_1+\mathfrak{m}_2$
can be explicitly presented as following. We choose $\langle (A,a),(B,b)\rangle_{\mathrm{bi}}=-\mathrm{tr}AB+ab$, for
$(A,a),(B,b)\in sp(n+1)\oplus\mathbb{R}$, in which $sp(n+1)=\{X\in\mathbb{H}^{(n+1)\times(n+1)}|
X+\overline{X}^t=0\}$. The subalgebra $\mathfrak{h}$ consists $(\mathrm{diag}(C,a\mathbf{i}),a)$ for all $C\in sp(n)$
and $a\in\mathbb{R}$. Then $\mathfrak{m}_0$ consists of
$(\mathrm{diag}(0_{n\times n},a\mathbf{i}),-a)$ for all $a\in \mathbb{R}$,
$\mathfrak{m}_1$ consists of $(\mathrm{diag}(0_{n\times n},b\mathbf{j}+c\mathbf{k}),0)$
for all $b,c\in\mathbb{R}$, and $\mathfrak{m}_2$ consists of
the pairs $(\left(
             \begin{array}{cc}
               0_{n\times n} & \mathrm{v} \\
               -\overline{\mathrm{v}}^t & 0 \\
             \end{array}
           \right),0)
$ for all column vectors $\mathrm{v}\in\mathbb{H}^n$.
Then direct calculations show the following:
\begin{eqnarray}
& &\label{013}
[\mathfrak{m}_0,\mathfrak{m}_0]=0,\quad
[\mathfrak{m}_0,\mathfrak{m}_1]\subset\mathfrak{m}_1,\quad
[\mathfrak{m}_0,\mathfrak{m}_2]\subset\mathfrak{m}_2\\
& &\label{012}
[u,v]=0\mbox{ for }u\in\mathfrak{m}_0\mbox{ and }v\in\mathfrak{m}_1+\mathfrak{m}_2\mbox{ if and only if }
u=0 \mbox{ or } v=0.
\end{eqnarray}

Firstly, we prove each $X\in\mathfrak{m}_0\backslash\{0\}$ is
a Randers equigeodesic for $G/H$.
Every metric operator must have the form
$\Lambda=c_0\mathrm{Id}|_{\mathfrak{m}_0}
\oplus c_1\mathrm{Id}|_{\mathfrak{m}_1}
\oplus c_2\mathrm{Id}|_{\mathfrak{m}_2}$, so
$[\Lambda(X),X]_\mathfrak{m}=[c_0 X,X]=0$,
and we see $X$ is a Riemannian equigeodesic vector.
It is obvious that $[X,\mathfrak{m}_0]=0\subset\mathfrak{h}$.
So $X$ is a Randers equigeodesic by Theorem \ref{main-thm-1}.

Nextly, we prove that each Randers equigeodesic $X$ for $G/H$ must belong to $\mathfrak{m}_0$. Denote $X=X_0+X_1+X_2$ with each $X_i\in\mathfrak{m}_i$.
By Theorem \ref{main-thm-1}, $[X,\mathfrak{m}_0]=[X_1+X_2,\mathfrak{m}_0]\subset\mathfrak{h}$.
However, by (\ref{013}),
$[X_1+X_2,\mathfrak{m}_0]\subset\mathfrak{m}_1+\mathfrak{m}_2$.
So $[X_1+X_2,\mathfrak{m}_0]=0$, and it implies $X_1=X_2=0$ by (\ref{012}).

This ends the case-by-case proof for Theorem \ref{main-thm-3}.

{\bf Acknowledgement}.
This paper is supported by Beijing Natural Science Foundation (No. 1222003), National Natural Science Foundation of China (No. 12001007, No. 12131012, No. 11821101),
Natural Science Foundation of Anhui province (No. 1908085QA03).

\end{document}